\documentclass[letterpaper, 10 pt, conference]{ieeeconf}
\IEEEoverridecommandlockouts                              
\usepackage{cite}
\usepackage{url}
\usepackage{array}
\usepackage{pgfplots} 
\pgfplotsset{compat=newest} 
\pgfplotsset{plot coordinates/math parser=false} 
\newlength\figureheight 
\newlength\figurewidth  
\usepackage{multirow}
\usepackage{amssymb}
\usepackage{amsmath}
\usepackage{graphicx}
\usepackage{epstopdf}
\usepackage{color}
\usepackage[english]{babel}
\usepackage{verbatim}
\usepackage{tikz}
\usepackage{booktabs}
\usepackage{mathtools}
\usepackage{siunitx}
\usepackage{accents}

\usepackage[utf8]{inputenc}
\usepackage[english]{babel}
\usepackage{bm}	
\usepackage{booktabs}
\usepackage{balance}

\usepackage{amsthm}
\newtheorem{theorem}{Theorem}
\newtheorem{remark}{Remark}
\usepackage[utf8]{inputenc}
\usepackage{algorithm}
\usepackage{esvect}
\usepackage[noend]{algpseudocode}
\newtheorem{corollary}{Corollary}

\graphicspath{{figures/}}
\renewcommand{\phi}{\varphi}

\newtheorem{definition}{Definition}[section]

\algdef{SE}[SUBALG]{Indent}{EndIndent}{}{\algorithmicend\ }%
\algtext*{Indent}
\algtext*{EndIndent}
\usepackage{soul}

\title{\LARGE \bf Robust Model Predictive Control  \\ with Adjustable Uncertainty Sets}
\author{Yeojun Kim, %
Xiaojing Zhang, %
Jacopo Guanetti, %
Francesco Borrelli 
\thanks{The authors are with the Department of Mechanical Engineering, University of California at Berkeley, CA, USA. E-mail: {\tt\footnotesize $\left\lbrace \textnormal{yk4938, xiaojing.zhang, jacopoguanetti, fborrelli}\right\rbrace$
@berkeley.edu}}%
}

\begin{document}
	\maketitle
	\thispagestyle{empty}
	\pagestyle{empty}
	\begin{abstract}
	In this paper, we present Robust Model Predictive Control (MPC) problems with \emph{adjustable uncertainty sets}. In contrast to standard Robust MPC problems with known uncertainty sets, we treat the uncertainty sets in our problems as additional decision variables. In particular, given a metric for adjusting the uncertainty sets, we address the question of determining the optimal size and shape of those uncertainty sets, while ensuring robust constraint satisfaction. The focus of this paper is on ensuring constraint satisfaction over an \emph{infinite horizon}, also known as persistent feasibility. We show that, similar as in standard Robust MPC, persistent feasibility can be guaranteed if the terminal set is an invariant set with respect to both the state of the system \emph{and} the adjustable uncertainty set. We also present an algorithm for computing such invariant sets, and illustrate the effectiveness of our approach in a cooperative adaptive cruise control application.
	

	
	\end{abstract}
	
	\section{Introduction} 
	Robust control concerns itself with designing controllers for uncertain systems \cite{dorato1987historical, petersen2014robust}. 
	In the field of Model Predictive Control (MPC), Robust MPC has been first introduced in \cite{campo1987robust,low2000robust,wu2001lmi}. Commonly, Robust MPC deals with problems where the uncertainty sets are known a priori. The objective in this case is to design a controller that is robust against all disturbance realizations from this uncertainty set.

    Recently, a new paradigm in robust control has been introduced where uncertainty sets are not assumed to be known, but can be \emph{adjusted} instead, i.e., their size is not fixed but to be determined. These kinds of problems are referred to as robust control problems with \emph{adjustable uncertainty sets} \cite{zhang2014selling, zhang2017robust}.
    
    Robust control problems with adjustable uncertainty sets arise in several applications: for example, in reserve provision problems, the adjustable uncertainty set is interpreted as a \emph{reserve capacity} that a system can offer to third parties and for which it receives (monetary) reward \cite{zhang2014selling,  VrettosTPS2016, zhang2017robust, Bitlisglioglu_TAC2017, ReyZhang2018}.
    In this case, the maximal reserve capacity can be computed by maximizing the size of the uncertainty set. The challenge when determining a reserve capacity is to ensure that, for every admissible reserve demand, the system is indeed able to meet this demand without violating its constraints.
    Another problem that can be formulated as a robust control problem with adjustable uncertainty sets arises in a robustness analysis for determining the limits of robustness of a given system (a system's resilience against disturbances).

	The existing literature on robust control with adjustable uncertainty sets has mainly focused on designing controllers over a finite horizon and deriving computationally tractable convex approximations \cite{zhang2014selling,VrettosTPS2016, zhang2017robust, Bitlisglioglu_TAC2017}. In this paper, we consider the \emph{infinite-horizon} case, where our objective is to design a controller that satisfies the constraints for all times.

    Specifically, we extend the work of \cite{zhang2017robust} to the infinite horizon case by employing ideas from standard Robust MPC. The resulting algorithm, called \emph{Robust MPC with Adjustable Uncertainty Sets (RMPC-AU)}, ensures satisfaction of state and input constraints for all future time steps, a notion known as persistent feasibility in the literature \cite{kouvaritakis2016model}. The contributions of this paper can be summarized as follows:
	\begin{itemize}
	 \item We introduce the notion of adjustable control invariant set and adjustable positive invariant sets, and show that these sets can be used to ensure persistent feasibility. We also provide algorithms for computing these sets.
	 \item Using ideas from standard Robust MPC, we show that, by appropriately parametrizing the control policies and uncertainty sets, computationally tractable reformulations can be derived.
	\end{itemize}

    We illustrate our approach on a cooperative adaptive cruise control (CACC) problem where the objective is to minimize the distance gap between vehicles while maximizing the uncertainty set for the future states of the other vehicles.
    Here, the challenge is that there is a trade-off between these two objectives and it is not straightforward to address it with standard methodologies \cite{lefevre2016learning, alam2014guaranteeing, Guanetti2018}.
    Using our approach, we can easily handle this challenge by simultaneously computing the input sequence and the allowable uncertainty set in a single optimization problem.
    Moreover, we can gain insights on the relationship between the distance tracking performance and the uncertainty set for the predicted states of the other vehicles.
     
    The remainder of this paper is organized as follows. Section \ref{sec:problemformulation} defines the RMPC-AU formulation and introduces the conditions which guarantee recursive feasibility for our problem. Section \ref{sec:Tractable} introduces the policy and the uncertainty set approximations and the terminal constraint formulation, and shows the final tractable formulation of the optimization problem for RMPC-AU. Section \ref{sec:casestudy} illustrates our method on a CACC example. Section \ref{sec:conclusions} concludes the work. The Appendix provides the derivations for tractable reformulation of our optimal control problem and the method to compute the adjustable control invariant set.
	
	\subsection{Notation} \label{sec:not}
    We denote by $x_{k\mid t}\in\mathbb R^{n_x}$ the state at time $k+t$, predicted at time $t$. 
    Furthermore, we define $\bm{x_{k\mid t}} :=[x_{0\mid t},x_{1\mid t},...,x_{k-1\mid t}]^{\top} \in\mathbb{R}^{kn_x}$.
	Given a set $\mathbb{X} \subseteq \mathbb{R}^{n_x}$, we denote by $\mathbb{X}^k$ its $k$-fold Cartesian product, i.e.,\  $\mathbb{X}^k := \underset{i=0}{\stackrel{k-1}{\bigtimes}} \mathbb{X} $. Finally, we denote by $\mathcal{P}(\mathbb{R}^{n_w})$ the power set of $\mathbb{R}^{n_w}$, and $\mathbb{N}_+$ is the set of non-negative integers.
	Also, $A \Rightarrow  B$ means that A implies B, i.e., if A is true then B must be true.
	\section{Problem Formulation} \label{sec:problemformulation}
	In this section we first review the finite-horizon robust optimal control problem with adjustable uncertainty sets from \cite{zhang2017robust}, before we  introduce the concept of Robust Model Predictive Control with adjustable uncertainty sets. 
	Also, $A \Rightarrow  B$ means that A implies B, i.e., if A is true then B must be true.
	
	\subsection{Robust  Control with Adjustable Uncertainty Sets}
	Following \cite{zhang2017robust}, we consider  discrete-time uncertain systems of the form
	\begin{equation}
	x_{k+1} = Ax_{k} + Bu_{k} + E w_{k},
	\label{linear_model}
	\end{equation}
	where $x_k \in \mathbb{R}^{n_x} $ is the state, $u_k \in \mathbb{R}^{n_u} $ is the input, and $w_k \in \mathbb{R}^{n_w}$ is the uncertain disturbance at time step $k$.
	We consider polytopic state and input constraints of the form
	\begin{subequations}\label{stateinputconstraints}
	\begin{align}
	x_k \in \mathbb{X} := \{ x \in \mathbb{R}^{n_x} : F_x x \leq f_x \}, \label{statecons} \\
	u_k \in \mathbb{U} := \{ u \in \mathbb{R}^{n_u} : F_u u \leq f_u \},
	\end{align}
	\end{subequations}
	where $F_x \in \mathbb{R}^{n_f \times n_x}, f_x \in \mathbb{R}^{n_f}, F_u \in \mathbb{R}^{n_g \times n_u}, f_u \in \mathbb{R}^{n_g}$ are known matrices. The uncertainty $w_k$ is assumed to belong to the \emph{adjustable} uncertainty set $\mathbb{W}_k$, i.e.,
	\begin{equation}\label{uncertaintycons}
	\begin{aligned}
	w_k \in \mathbb{W}_k \subseteq \mathbb{R}^{n_w}.
	\end{aligned}
	\end{equation}
	Here, $\mathbb{W}_k$ is called ``adjustable" because its size is not defined a priori, but to be determined later. 
	Given the planning horizon $N$, we denote by $\phi_{k\mid t}(\bm{u_t, w_t})$ the state at time $t+k$  predicted at time $t$ using the model \eqref{linear_model} subject to the inputs $\bm{u_t}=[u_{0\mid t}, u_{1\mid t}, ..., u_{N-1\mid t}]$ and the uncertainties $\bm{w_t} = [w_{0\mid t}, w_{1\mid t}, ..., w_{N-1\mid t}]$. 
	
    Our goal is to determine the optimal input policy and the uncertainty sets, over the  horizon $N$, such that constraints \eqref{stateinputconstraints} are satisfied.
We consider state feedback policies $\bm{\pi_t}(\cdot) := [\pi_{0\mid t}(\cdot),\pi_{1\mid t}(\cdot),...,\pi_{N-1\mid t}(\cdot)]$, where $\pi_{k\mid t}(\cdot) : \mathbb{R}^{n_x} \to \mathbb{R}^{n_u}$, such that the input at time step $k$ is given by $u_{k\mid t} =\pi_{k\mid t}(x_{k\mid t})$.

    Given a metric $\rho(\cdot) : \mathcal{P}(\mathbb{R}^{n_w}) \to \mathbb{R}$ for adjusting the uncertainty set, the control objective is to minimize a ``nominal" cost $J(\cdot,\cdot)$ while maximizing the uncertainty set with respect to $\rho(\cdot)$. 
    For instance, $\rho(\cdot)$ can represent the volume of the uncertainty set, but is in general problem-dependent \cite{zhang2017robust}.
	Therefore, the finite time constrained optimal control problem with adjustable uncertainty sets at time $t$ is given by
    \begin{equation}
    \label{bare_optimal_probem}
    \begin{array}{ll}
    \displaystyle\min_{\bm{\pi_t}(\cdot), \mathcal{W}_{t} } &  \Big\{ \underset{ \bm{w_t} \in \mathcal{W}_{t}} {\textnormal{max}} \, J(\bm{\pi_{t}(\bm{x_t}), w_t}) \Big\} - \displaystyle\sum_{k=0}^{N-1}\lambda_k \rho (\mathbb{W}_{k\mid t}) \\
    \hphantom{.}\textrm{s.t.} & \\
    	& x_{k+1\mid t} = Ax_{k\mid t} +B \pi_{k\mid t} (x_{k\mid t}) +Ew_{k\mid t}, \\
    	& x_{k \mid t} \in \mathbb{X}, \, \pi_{k\mid t} ( x_{k\mid t} ) \in \mathbb{U}, \, \forall \, w_{k\mid t} \in \mathbb{W}_{k\mid t}, \\
    	& \forall \, k = 0,...,N-1,\\
        &  x_{0\mid t} = x(t), \,\, (x_{N\mid t},\mathbb{W}_{N-1\mid t}) \in \mathbb{F},
    \end{array}
    \end{equation}
    where $\mathcal{W}_{t} := \mathbb{W}_{0\mid t} \times \mathbb{W}_{1\mid t} \times \dots \times \mathbb{W}_{N-1\mid t}$; $J(\bm{\pi_{t}(\bm{x_t}), w_t})  := \ell_f( \phi_{N\mid t}(\bm{\pi_{t}(\bm{x_t}), w_t})  ) + \sum_{k=0}^{N-1} \ell(  \phi_{k\mid t}( \bm{\pi_{t}(\bm{x_t}), w_t} ) )$ and $\ell$ and $\ell_f$ are linear stage costs. $x(t)$ and $\lambda_k \geq 0$ are the initial condition and the weighting factors, respectively. 


	
	Notice that problem \eqref{bare_optimal_probem} imposes a so-called terminal constraint 
	\begin{equation} \label{term_cons}
	 (x_{N\mid t},\mathbb{W}_{N-1 \mid t}) \in \mathbb F \subseteq \mathbb{X} \times \mathcal P (\mathbb{R}^{n_w}).
	\end{equation}
	As we will see, the terminal set $\mathbb{F}$ plays a crucial role in ensuring persistent feasibility of our proposed controller, and is discussed in the following sections in detail.
	

	\subsection{Robust MPC with Adjustable Uncertainty Set} \label{sec:mpc_au}
	
	The Robust MPC with Adjustable Uncertainty Sets (RMPC-AU) solves, at time $t$, the finite time constrained optimal control problem \eqref{bare_optimal_probem}. Let
	\begin{subequations} \label{eq:sol}
    \begin{align}
	 \bm{\pi^{\ast}_t}(\cdot) &= [\pi^{\ast}_{0 \mid t}(\cdot), \pi^{\ast}_{1 \mid t}(\cdot) , ... , \pi^{\ast}_{N-1 \mid t}(\cdot) ], \\
	 \mathcal{W}^{\ast}_t &= [\mathbb{W}^{\ast}_{0\mid t},\, \mathbb{W}^{\ast}_{1\mid t}, ..., \, \mathbb{W}^{\ast}_{N-1\mid t}]
	\end{align}
	\end{subequations}
	be the solution of the optimal control problem~\eqref{bare_optimal_probem}.
	Then, the first input $\pi^{\ast}_{0 \mid t} (x_t) $ is applied to the system.
	At the next time step $t + 1$, a new optimal control problem in the form of \eqref{bare_optimal_probem} based on new measurements of the state is solved over a shifted horizon, yielding a \textit{moving} or \textit{receding horizon} control strategy. The control law is given by
    \begin{equation} \label{eq:opt_sol}
	    u_t = \pi^{\ast}_{0 \mid t}(x_t),
	\end{equation}
	and the closed loop system can be written as
	\begin{equation} \label{eq:cl_sys}
	    x_{t+1} = Ax_t + B\pi^{\ast}_{0 \mid t}(x_t) + E w_t,
	\end{equation}
    where $w_t \in \mathbb{W}^{\ast}_{0\mid t}$.
    
	\subsection{Recursive Feasibility of RMPC-AU} \label{sec:feasibility}
	Our objective is to design a controller which ensures state and input constraint satisfaction at all times. 
	In this section we formally define the notion of recursive feasibility for RMPC-AU \eqref{bare_optimal_probem}-\eqref{eq:cl_sys}, before we provide sufficient conditions which guarantee it.
	\begin{definition}[Recursive feasibility for RMPC-AU] \label{def_recur_feas}
		 RMPC-AU \eqref{bare_optimal_probem}-\eqref{eq:cl_sys} is said to be \textit{recursively feasible}, if the existence of a solution $(\bm{\pi^{\ast}_0},\mathcal{W}^{\ast}_0)$ for \eqref{bare_optimal_probem} at  time  $t=0$ implies feasibility of MPC-AU \eqref{bare_optimal_probem}-\eqref{eq:cl_sys} for all time $t > 0$.
	\end{definition}
	
    In general, recursive feasibility is not guaranteed without imposing additional conditions on the terminal constraint $\mathbb F$. Similar as in standard Robust MPC \cite{kouvaritakis2016model}, the recursive feasibility of RMPC-AU can be achieved by choosing the terminal set $\mathbb F$ to be an invariant set.
    \begin{definition}[Adjustable Control Invariant Set] \label{C}
		A set $ \mathbb{C}_{\textnormal{adj}}  \subseteq \mathbb{X} \times \mathcal P(\mathbb{R}^{n_w})$ is said to be an \emph{adjustable control invariant set} for the system \eqref{linear_model} subject to the constraints \eqref{stateinputconstraints} and the uncertainty \eqref{uncertaintycons}, if
		\begin{equation}\label{eq:adjCtrlInvSet}
		\begin{array}{llll}
		\forall (x_t, \mathbb{W}_{t-1}) \in  \mathbb{C}_{\textnormal{adj}} ~ \Rightarrow \, \exists \, u_t \in \mathbb{U}, \, \mathbb{W}_{t} \subseteq \mathbb{R}^{n_w}\\
		\qquad  (x_{t+1}, \mathbb{W}_{t}) \in \mathbb{C}_{\textnormal{adj}},~\forall w_{t} \in  \mathbb{W}_{t}, \, t\in \mathbb{N}_+.
		\end{array}
		\end{equation}
	\end{definition}
	Notice that Definition~\ref{C} is a generalization of the standard robust control invariant set, which can be recovered by fixing $\mathbb{W}_{t-1}=\mathbb W_t$ in \eqref{eq:adjCtrlInvSet}.
Furthermore, when $\mathbb W_{t-1}=\mathbb W_t = \{0\}$, \eqref{eq:adjCtrlInvSet} recovers the standard control invariant set \cite{borrelli2017predictive}.
	We now have the following result:
	\begin{theorem} \label{recur_theo}
		Consider  RMPC-AU \eqref{bare_optimal_probem}-\eqref{eq:cl_sys} with $N\geq1$. If the terminal set $\mathbb{F}$ in \eqref{bare_optimal_probem} is an adjustable control invariant set for the system \eqref{linear_model}-\eqref{uncertaintycons}, then the RMPC-AU is recursively feasible. 
	\end{theorem} 
	\begin{proof}
	We prove Theorem \ref{recur_theo} by induction, showing that taking the terminal set $\mathbb{F}$ as an adjustable control invariant set is a sufficient condition for persistent feasibility of RMPC-AU \eqref{bare_optimal_probem}-\eqref{eq:cl_sys}. 
	
	First, consider the optimal control problem \eqref{bare_optimal_probem} at time $t$ with its optimal solution $(\bm{\pi_t}^{\ast}(\cdot), \, \mathcal W^{\ast}_t)$ as described in \eqref{eq:opt_sol}. 
	At the next time step $t+1$, for all $w_t\in\mathbb W_t$,
	the optimal control problem \eqref{bare_optimal_probem} is guaranteed to have at least one feasible solution with the shifted input policy and uncertainty set and the added solution, $(\hat{u},\hat{\mathbb{W}})$, at the last step. Here, $\hat{u}$ and $\hat{\mathbb{W}}$ are the feasible input and the uncertainty set, respectively, which are guaranteed to exist by Definition II.2. Explicitly, this feasible solution at time $t+1$ can be written as:
	\begin{subequations}
    \begin{align}
	 \bm{\pi_{t+1}}(\cdot) &= [\pi^{\ast}_{1 \mid t}(\cdot), \pi^{\ast}_{2 \mid t}(\cdot) , ... , \hat{u} ], \\
	  \mathcal{W}_{t+1} &= [\mathbb{W}^{\ast}_{1\mid t},\, \mathbb{W}^{\ast}_{2\mid t}, ..., \, \hat{\mathbb{W}}].
	\end{align}
	\end{subequations}
 Therefore, by induction, RMPC-AU is recursively feasible.
	\end{proof}
    While Theorem 1 establishes recursive feasibility, solving the optimization problem \eqref{bare_optimal_probem} is computationally intractable since $(i)$ the optimization is performed over general feedback policies, $(ii)$ the optimization of the uncertainty set is performed over arbitrary subsets of $\mathbb R^{n_w}$, and $(iii)$ the constraints must be satisfied robustly for every uncertainty realization \cite{zhang2017robust}. In the following,  we present approximations of RMPC-AU \eqref{bare_optimal_probem}--\eqref{eq:cl_sys} that are computationally tractable while ensuring recursive feasibility.

    \section{Tractable Approximation of RMPC-AU} \label{sec:Tractable}
    In this section we provide a tractable convex formulation for the optimal control problem in \eqref{bare_optimal_probem}.

	
	\subsection{Uncertainty Set and Policy Approximation} \label{uncerainty_policy_back}
	Following \cite{zhang2017robust}, we consider uncertainty sets that are obtained as the affine transformation of a so-called primitive set $\mathbb{S}$, i.e., 
	\begin{align} \label{eq:unceraintyset_prim}
	\mathbb{W}_t  = Y_t \mathbb{S} +y_t, 	\text{ where } (Y_t,\,y_t) \in \mathbb{Y} \subseteq \mathbb R^{n_w \times n_s} \times \mathbb R^{n_w},
	\end{align}
	and 
	\begin{align} \label{primitiveset}
	\mathbb{S} := \{ s \in \mathbb{R}^{n_s} \, : \, G s \leq  g  \}
	\end{align}
	is a fixed polytope with given matrices $G \in \mathbb{R}^{l \times n_s}$ and $g \in \mathbb{R}^{l}$. The set $\mathbb Y$ in \eqref{eq:unceraintyset_prim} allows us to control the shape of the uncertainty set, see \cite[Section 3.2]{zhang2017robust} for details. Given the planning horizon $N$, we define by $\bm{Y_t}$ and $\bm{y_t}$ the compact forms of the $N$ time steps forward consecutive $Y_{k\mid t}$ and $y_{k\mid t}$ given the time $t$, respectively; i.e., $\bm{Y_t} := \textnormal{diag}(Y_{0\mid t},Y_{1\mid t}, \dots, Y_{N-1\mid t})$ and $\bm{y_t} := [y_{0\mid t},y_{1\mid t}, \dots, y_{N-1\mid t}]^{\top}$.

	It is well-known that optimizing over general feedback policies is computationally intractable. 
	To overcome this issue, we adopt the affine disturbance feedback policy of \cite{goulart2006optimization}, which is given by
		\begin{equation}  \label{causalpolicy}
	    u_{k\mid t} = \Tilde{\pi}_{k\mid t}(\bm{s_{k\mid t}}) := p_{k\mid t}+ \sum_{j=0}^{k-1}P_{k,j\mid t}s_{j\mid t},
		\end{equation}
	where $p_{k|t}\in\mathbb R^{n_u}$ and $P_{k,j|t}\in\mathbb R^{n_u \times n_s}$ are parameters.
	Note that the control policy \eqref{causalpolicy} is defined with respect to $s_k$, and not with respect to $w_k$.
	We denote by $\bm{{\Tilde{\pi}}_{t}}(\bm{s_{N-1\mid t}})$ the concatenated form of the $N$ time steps forward consecutive policies at time $t$; i.e., $\bm{{\Tilde{\pi}}_{t}}(\bm{s_{N-1\mid t}}) := [\Tilde{\pi}_{0\mid t}(\bm{s_{N-1\mid t}}), \, \dots, \, \Tilde{\pi}_{N-1\mid t}(\bm{s_{N-1\mid t}})]^{\top} = \bm{p_t} + \bm{P_t}\bm{s_{N-1\mid t}}$, where $\bm{P_t} \in \mathbb{R}^{Nn_u \times Nn_s}$ is a lower triangular block matrix whose elements are defined by  $P_{k,j\mid t}$ and $\bm{p_t} := [p_{0\mid t}, \, \dots, \, p_{N-1\mid t}]^{\top} \in \mathbb{R}^{Nn_u}$.
	
	In the remainder of this paper, we assume that the uncertainty set can be represented as in \eqref{eq:unceraintyset_prim}-\eqref{primitiveset}. Furthermore, by considering the policies \eqref{causalpolicy}, problem \eqref{bare_optimal_probem} can be rewritten as
    \begin{equation}
    \label{linear_adjusted_optimal_probem}
    \begin{array}{ll}
    \displaystyle\min & \Big\{ \underset{ \bm{s_t} \in \mathbb{S}^N } {\textnormal{max}} \, \bm{c}^{\top}\bm{{\Tilde{\pi}}_{t}(\bm{s_{N\mid t}})}  \Big\} - \displaystyle\sum_{k=0}^{N-1}\lambda_k \rho (Y_{k\mid t} \mathbb{S} + y_{k\mid t}) \\
    \hphantom{.}\textrm{s.t.} & \\
    & x_{k+1\mid t} = Ax_{k\mid t} +B\Tilde{\pi}_{k\mid t}(\bm{s_{k\mid t}}) +E w_{k\mid t}, \\
    & \Tilde{\pi}_{k\mid t}(\bm{s_{k\mid t}}) = p_{k\mid t}+ \sum_{j=0}^{k-1}P_{k,j\mid t}s_{j\mid t},  \\
    & w_{k\mid t} = Y_{k\mid t} s_{k\mid t} + y_{k\mid t},\,\, (Y_{k\mid t},\,y_{k\mid t}) \in \mathbb{Y} \\
    & x_{k \mid t} \in \mathbb{X}, \,\, \Tilde{\pi}_{k\mid t}(\bm{s_{k\mid t}}) \in \mathbb{U}, \,\, \forall \, s_k \in \mathbb S \\
    & \forall \, k = 0 , 1 , \dots , N-1 , \\
    & x_{0\mid t} = x(t) ,\; (x_{N\mid t},\,Y_{N-1\mid t} \mathbb{S} + y_{N-1\mid t}) \in \mathbb{F},
    \end{array}
    \end{equation}
    with $( \bm{P_t}, \, \bm{p_t}, \, \bm{Y_t}, \, \bm{y_t})$ as the decision variables. $\bm{c}\in\mathbb R^{Nn_u}$ is a matrix constructed to represent the linear stage cost from the problem data, see e.g. \cite{goulart2006optimization} for details on the construction.

	\subsection{Terminal Set Approximation} \label{uncerainty_policy_back}
	
	\subsubsection{Adjustable Positive Invariant Set}
	
	
	Recall that we restricted our input to be in the form \eqref{causalpolicy} to attain computational tractability. Unfortunately, this restriction also implies that Theorem~1 cannot be directly applied because it guarantees the existence of an input that may not depend on the uncertainty in a linear fashion.
	
	\begin{remark} \label{rm1}
	 If the uncertainty set \eqref{eq:unceraintyset_prim} is assumed to always include $\{0\}$, then Theorem 1 can guarantee recursive feasibility of the reformulated RMPC-AU \eqref{linear_adjusted_optimal_probem},\eqref{eq:sol}-\eqref{eq:cl_sys}, at least with $\mathbb W_{k\mid t} = \{0\}$ for $k=0,\dots,N-1$. This is because the control policy \eqref{causalpolicy} becomes equivalent to the deterministic input actions when the uncertainty set is $\{0\}$.
	\end{remark}
	
	This motivates the following definition:
	\begin{definition}[Adjustable Positive Invariant Set] \label{O}
		A set $ \mathbb{O}_{\textnormal{adj}} \subseteq \mathbb{X} \times \mathcal P (\mathbb{R}^{n_w})$ is said to be an \textit{adjustable positive invariant set} for the uncertain autonomous system $x_{t+1} = g( x_{t}, w_{t})$ subject to the constraints \eqref{statecons} and the uncertainty \eqref{uncertaintycons}, if
		\begin{equation}
		\begin{array}{llll}
		 \forall (x_t, \mathbb{W}_{t-1}) \in  \mathbb{O}_{\textnormal{adj}} ~ \Rightarrow \,  \exists \, \mathbb{W}_{t} \subseteq \mathbb{R}^{n_w} \\ 
		 \qquad \qquad \, (x_{t+1},\mathbb{W}_{t}) \in \mathbb{O}_{\textnormal{adj}},~\forall w_t \in  \mathbb{W}_t, \, t\in \mathbb{N}_+.
		\end{array}
		\end{equation}
	\end{definition}
	Note that an autonomous system  $x_{t+1}=g( x_{t}, w_{t})$ can be achieved from the control system \eqref{linear_model} by fixing a feedback control policy.
	
    
    Now, we are in place to state the following result.
	\begin{corollary} \label{corollary}
	    Consider the RMPC-AU \eqref{linear_adjusted_optimal_probem}, \eqref{eq:sol}-\eqref{eq:cl_sys} for $N\geq1$. If the terminal set $\mathbb{F}$ in \eqref{bare_optimal_probem} is an adjustable positive invariant set for the system \eqref{linear_model}-\eqref{uncertaintycons} in a closed loop with an affine state feedback law $u_t = Kx_t + b $, where $K \in \mathbb R^{n_u \times n_x}$ and $b\in \mathbb R^{n_u}$ are given, then the RMPC-AU \eqref{linear_adjusted_optimal_probem}, \eqref{eq:sol}-\eqref{eq:cl_sys} is recursively feasible.
	\end{corollary} 
	\begin{proof}
	By Definition II.2 and III.1, an adjustable positive invariant set for a control system with a fixed policy is an adjustable control invariant set for that control system. Therefore, it follows from Theorem 1 that the RMPC-AU \eqref{bare_optimal_probem}-\eqref{eq:cl_sys} is recursively feasible when its terminal set is an adjustable positive invariant set. 
	
	It is known that the class of admissible affine state feedback policies is equivalent to the class of admissible affine disturbance feedback policies \cite{goulart2006optimization}. Moreover, any affine disturbance feedback policy can be represented as the feedback policy \eqref{causalpolicy} when the uncertainty set is approximated by \eqref{eq:unceraintyset_prim} and \eqref{primitiveset}, as proven in \cite{zhang2017robust}. 
	Therefore, the existence of the affine state feedback policy guarantees the existence of a feedback policy \eqref{causalpolicy} with the uncertainty approximation \eqref{eq:unceraintyset_prim} and \eqref{primitiveset}.
	Hence, the RMPC-AU \eqref{linear_adjusted_optimal_probem}, \eqref{eq:sol}-\eqref{eq:cl_sys}, which is subject to the feedback policy \eqref{causalpolicy}, is recursively feasible.
	\end{proof}
    
    \subsubsection{Computation of Adjustable Positive Invariant Set} \label{computing_pos_set}
    In this section, we present an algorithm for computing an adjustable positive invariant set $\mathbb O_{\textnormal{adj}}$. 
    In particular, we propose the method to compute the size of $\mathbb O_{\textnormal{adj}}$ through which we can ensure recursive feasibility (Corollary~1).
    
     
    
	Consider the affine state feedback policy $u_k = K x_k + b$ for the system dynamics \eqref{linear_adjusted_optimal_probem}. Using the uncertainty approximation \eqref{eq:unceraintyset_prim}, \eqref{primitiveset}, the closed loop system dynamics at time step $k$ can be written as: 
	\begin{align} \label{eq:closed_sys}
	       x_{k+1} = & (A+BK)x_k + E(Y_k s_k + y_k )+ Bb \nonumber \\
	               = & (A+BK)x_k + E\mathcal{S}\textnormal{vec}(Y_k)+ Ey_k + Bb,
    \end{align} 
	    where $\mathcal{S} := \textnormal{diag}(s_k^T, ..., s_k^T) \in \mathbb{R}^{n_w \times n_wn_s}$ and $\textnormal{vec}(\cdot)$ is a linear transformation which converts the matrix into a column vector in a column-wise manner. This system dynamics can be reformulated as
	    \begin{align} \label{eq:parametricuncertaintymodel_O}
            z_{k+1} :=& \left[
            \begin{array}{c}
            x_{k+1}  \\
            \hline
            \textnormal{vec}(Y_{k+1}) \\
            \hline
            y_{k+1}
            \end{array}\right] = g(z_k,\,s_k) :=  \\
            =& \left[\begin{array}{c|c|c}
            A + BK & E \mathcal{S} & E   \\
            \hline
            0 & I & 0 \\
            \hline
            0 & 0 & I
            \end{array}\right]
            \left[
            \begin{array}{c}
            x_{k}  \\
            \hline
            \textnormal{vec}(Y_k)  \\
            \hline
            y_k
            \end{array}\right] 
            +
            \left[
            \begin{array}{c}
            Bb  \\
            \hline
            0  \\
            \hline
            0
            \end{array}\right]. \nonumber
	    \end{align}
    The robust positive invariant set of the model \eqref{eq:parametricuncertaintymodel_O} represents the adjustable positive invariant set of our system \eqref{linear_model}-\eqref{uncertaintycons} with the uncertainty set approximation \eqref{eq:unceraintyset_prim}. We also denote by $\mathcal{Z}:=\mathbb{X} \times \mathbb{M}$ the feasible set of \eqref{eq:parametricuncertaintymodel_O} where $\mathbb{M} := \{ (\textnormal{vec}(Y_k),\, y_k) : (Y_k,\,y_k) \in \mathbb{Y} \subseteq  \mathbb{R}^{n_w n_s + n_w} \}$. Therefore, $z_k$ has dimension $n:=n_x+ n_w n_s + n_w$. 
    
    Finally, the robust control invariant set for the system \eqref{eq:parametricuncertaintymodel_O} constrained to the feasible set $\mathcal Z$ represents the adjustable positive invariant set for \eqref{linear_model}-\eqref{uncertaintycons} with the uncertainty set approximation \eqref{eq:unceraintyset_prim}.

	In order to compute the robust positive invariant set for \eqref{eq:parametricuncertaintymodel_O}, we first need the following definition from \cite{borrelli2017predictive}.
	
	\begin{definition}[Robust Precursor Set] \label{def_robust_precursor}
	For the autonomous system \eqref{eq:parametricuncertaintymodel_O}, we define the robust precursor set of the set $\mathcal Z$ as
	\begin{equation}
	    \textnormal{Pre}(\mathcal Z, \, \mathbb S) = \{ z\in\mathbb R^n : g(z,\,s)  \in \mathcal Z, \,\forall\, s \in \mathbb S \}
	\end{equation}
	\end{definition}
	$\textnormal{Pre}(\mathcal Z, \, \mathbb S)$ defines the set of states $z$ which can be driven into the target set $\mathcal Z$ in one step for all uncertainty $s\in\mathbb S$.
	Note that the robust precursor set of the system \eqref{eq:parametricuncertaintymodel_O} can be computed using \cite[Lemma 10.1]{borrelli2017predictive}. 
	Algorithm \ref{algo_O_inf} outlines the procedure for computing the adjustable robust control invariant set.
	If Algorithm~\ref{algo_O_inf} converges, then $\mathbb{O}_{\textnormal{adj}}$ is an adjustable positive invariant set.
\begin{algorithm}
\caption{Computation of $\mathbb{O}_{\textnormal{adj}}$}\label{algo_O_inf}
\begin{algorithmic}[1]
\State \textbf{Input} \textnormal{System model \eqref{eq:parametricuncertaintymodel_O}, $\mathcal{Z}$, $\mathbb{S}$}
\State \textbf{Output} \textnormal{$\mathbb{O}_{\textnormal{adj}}$}
\Indent
    \State $\Omega_0 \leftarrow \mathcal{Z}, \, k \leftarrow -1$
    \State \textbf{Repeat}
    \Indent
        \State $k \leftarrow k+1$, \, 
        \State $\Omega_{k+1} \leftarrow \textnormal{Pre}(\Omega_{k},\mathbb S) \, \cap \, \Omega_k$
    \EndIndent
    \State \textbf{Until} $\Omega_{k+1} = \Omega_k$
    \State $\mathbb{O}_{\textnormal{adj}} \leftarrow \Omega_k$
\EndIndent
\end{algorithmic}
\end{algorithm}

\begin{remark}
The adjustable positive invariant set computed by Algorithm \ref{algo_O_inf} can be used to find the largest constant uncertainty set which the system can indefinitely tolerate without violating its constraints. This is because the uncertainty set $\mathbb W_k =Y_k\mathbb S + y_k$ is time-invariant in the dynamics \eqref{eq:parametricuncertaintymodel_O}. 
\end{remark}

    Algorithm \ref{algo_O_inf} returns us, if it converges, a polytopic description of $\mathbb{O}_{\textnormal{adj}}$ with  $n_t$ linear constraint, which we can directly integrate into the optimization problem as
    \begin{equation} \label{term_repren}
        F  [x_k; \textnormal{vec}(Y_k);y_k] \leq f
    \end{equation}
    where $F \in \mathbb{R}^{n_t \times (n_x+n_w)}$ and $f \in \mathbb{R}^{n_t}$.

    \subsection{Tractable Reformulation}
    

    By combining the uncertainty set approximation \eqref{eq:unceraintyset_prim}, the feedback policy \eqref{causalpolicy}, and the polytopic terminal constraints \eqref{term_repren}, problem \eqref{linear_adjusted_optimal_probem} can be rewritten as 
    
    \begin{equation}
    \label{final_formulation}
    \begin{array}{ll}
    \displaystyle\min & \tau_t - \displaystyle\sum_{k=0}^{N-1}\lambda_k \rho (Y_{k\mid t} \mathbb{S} + y_{k\mid t}) \\
    \hphantom{.}\textrm{s.t.} & \\
    & \tau_t \in \mathbb{R}, \\
    & \bm{P_t} \in \mathbb{R}^{Nn_u \times Nn_s}, \, \bm{p_t} \in \mathbb{R}^{Nn_u}, \, (\bm{Y_t},\bm{y_t}) \in \mathcal{Y}\,, \\
    & \bm{\mu_t} \in \mathbb{R}^{Nl}, \, \bm{\Lambda_t} \in \mathbb{R}^{N(n_f+n_g) \times Nl}, \, \bm{\Gamma_t} \in \mathbb{R}^{n_t \times Nl}, \\
    & \bm{\mu_t} \geq  0, \quad \bm{\Lambda_t} \geq 0, \quad \bm{\Gamma_t} \geq 0, \\
    & \bm{c}^{\top} \bm{p_t} + \bm{\mu_t}^{\top} \bm{g_t} \leq \tau_t, \quad \bm{\mu_t}^{\top} \bm{G} = \bm{c_t}^{\top} \bm{P_t} \,, \\
    & \bm{C_t} \bm{p_t} + \bm{D_t} \bm{y_t} + \bm{\Lambda_t} \bm{g} \leq \bm{d_t}, \\
    & \bm{\Lambda_t} \bm{G} = \bm{C_t} \bm{P_t} + \bm{D_t} \bm{Y_t} , \\
    & \bm{\Gamma_t} \bm{g} + F[M_t\bm{p_t} + N_t \bm{y_t};\, \textnormal{vec}(Y_{N-1\mid t});\, y_{N-1\mid t}  ]  \leq  f, \\
    & \bm{\Gamma_t} \bm{G} = F[M_t \bm{P_t} + N_t \bm{Y_t};\,0;\,0] ,
    \end{array}
    \end{equation}
    where $(\tau_t, \, \bm{P_t}, \, \bm{p_t}, \, \bm{Y_t}, \, \bm{y_t}, \, \bm{\mu_t}, \, \bm{\Lambda_t}, \, \bm{\Gamma_t})$ are the decision variables;
    \begin{align}
    \mathcal{Y} : = \Big\{ &(\bm{Y_t},\bm{y_t}) :\{(Y_{k\mid t}, y_{k\mid t}) \in \mathbb{Y}\}_{k=0}^{N-1}, \nonumber\\
    &\bm{Y_t} = \textnormal{diag}[Y_{0\mid t}, ... , Y_{N-1\mid t}] \in \mathbb{R}^{Nn_w \times Nn_s}, \nonumber \\
					 &\bm{y_t} = [y_{0\mid t}, ..., y_{N-1\mid t} ]^{\top} \in \mathbb{R}^{Nn_w}  \Big\}; \nonumber
	\end{align}
	$\bm{G} := \textnormal{diag}(G,...,G) \in \mathbb{R}^{Nl \times Nn_s}$ ; $\bm{g}:= [g,...,g]^{\top} \in \mathbb{R}^{Nl}$;
	$(\bm{C_t},\, \bm{D_t},\, \bm{d_t},\, M_t,\, N_t)$ are given in Appendix A. Problem \eqref{final_formulation} is a convex optimization problem with a finite number of constraints and decision variables; therefore, it can be solved using off-the-shelf solvers such as GUROBI \cite{gurobi} or MOSEK \cite{mosek2010mosek}.

\section{Example: Vehicle Platooning} \label{sec:casestudy}

	In this section we illustrate the proposed RMPC-AU on a cooperative adaptive cruise control (CACC) application. 
	We consider a group (or \emph{platoon}) of two vehicles featuring bi-directional Vehicle-to-Vehicle (V2V) communication. The \emph{ego}-vehicle is positioned in the back, and receives the front vehicle's velocity, $v_f$. In return, the ego vehicle sends the maximum acceleration or deceleration which the front vehicle must not exceed in order to avoid a rear collision and, on the other hand, loss of contact because the spacing has become excessive.
	The front vehicle can then enforce this constraint on the acceleration/deceleration as done in \cite{lefevre2016learning,turri2017cooperative}, thus ensuring safety in the platoon operation. Our platoon setup is depicted in Fig.~\ref{fig:sheme}. 
	
	The control objective of our problem is to minimize the  tracking distance while maximizing the magnitude of the acceleration/deceleration that the front vehicle can undertake. However, there is a trade-off between these two factors: the framework introduced in the previous sections allows us to systematically address this trade-off. For instance, if the objective is energy saving, then tracking distance should be minimized, thereby reducing the aerodynamic drag on the rear vehicle. This will result in smaller bounds on the acceleration/deceleration of the front vehicle. If the objective is to allow more freedom to the movements of the front vehicle, then its maximum acceleration/deceleration should be maximized, which will result in a larger distance gap. 
	In our optimization problem, the adjustable uncertainty set is interpreted as a constraint set for the input of the front vehicle 

	We consider the longitudinal dynamics of the ego vehicle when following the front vehicle. We assume a point mass model with two states: $d$, the distance between the ego and the front vehicle, and $\Delta v$, the difference between the ego vehicle's velocity $v_{\textnormal{ego}}$ and the front vehicle's velocity $v_f$; i.e. $\Delta v = v_f - v_{\textnormal{ego}}$.
	There is one input, $u$, which represents the acceleration of the ego vehicle, and one additive disturbance, $w$, which represents the acceleration or deceleration of the front vehicle. The model can be written in the form of \eqref{linear_model} with
	\begin{align} \label{eq:ex_model}
	A =
	\begin{bmatrix}
	1            &        \Delta t \\
	0            &         1 
	\end{bmatrix}, 
	\, B = 
	\begin{bmatrix}
	0              \\
	-\Delta t
	\end{bmatrix},
	\, E = 
	\begin{bmatrix}
	0    \\
	\Delta t
	\end{bmatrix}.
	\end{align}
	where $\Delta t = 0.2\si{s}$ represents the discretization time. The uncertainty $w_k$ is bounded by some symmetric uncertainty set around zero, i.e. $w_k \in \mathbb W_k \subseteq \mathbb{R}$ where each $\mathbb W_k := \{ w \, :\, |w| \leq y_k \in \mathbb{R}_{\ge 0} \}$. $y_k$ represents the maximum acceleration and deceleration of the front vehicle. Given this shape of the uncertainty sets, $y_k$ quantifies the size of $\mathbb W_k$. Moreover, we enforce that the uncertainty set has a constant size over the horizon, i.e., $y_{t} =y_{0\mid t}=\dots=y_{N-1\mid t}$.

\begin{figure}
	\centering
	\includegraphics[width=\linewidth,height=\textheight,keepaspectratio]{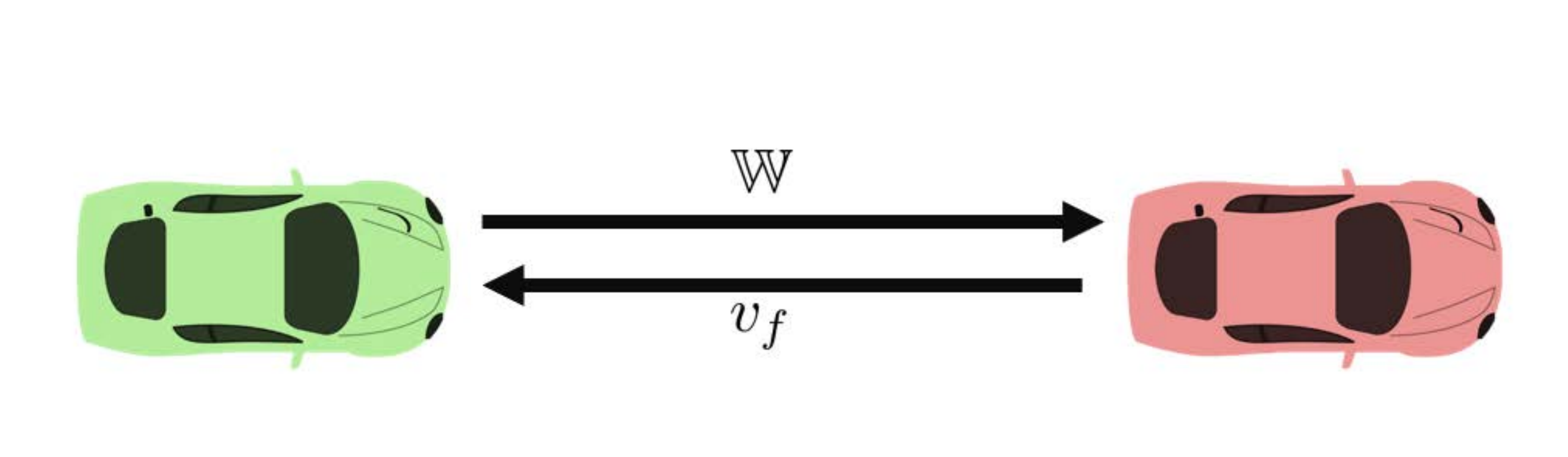}
	\caption{Communication schematic between the ego vehicle (left) and the front vehicle (right)}
	\label{fig:sheme} 
\end{figure} 

The optimal control problem at time $t$ can be formulated as 
\begin{equation}
\label{ex_optimal_probem}
\begin{array}{ll}
\displaystyle\min_{{\bm{\pi_t}(\cdot), \, y_{t} }} &\Big\{ \underset{ \bm{w_t} \in \mathbb{W}^N } {\textnormal{max}} \, \displaystyle\sum_{k=0}^{N} [ 1, 0] x_{k\mid t} \Big\} - \lambda y_{t}  \\
\hphantom{.}\textrm{s.t.} & \\
& x_{k+1\mid t} = Ax_{k\mid t} +B \pi_{k\mid t} (x_{k\mid t}) +E w_{k\mid t}, \\
& \begin{bmatrix}
	10             \\
	-5
	\end{bmatrix} \leq x_{k\mid t} \leq \begin{bmatrix}
	20              \\
	5
	\end{bmatrix}, -10 \leq \pi_{k\mid t} (x_{k\mid t}) \leq 10, \\
    & \forall \, w_{k\mid t} \in [-y_t,y_t ], \\
& \forall \, k = 0 , 1 , \dots , N-1 , \\
& x_{0\mid t} = x(t) ,\;  (x_{N\mid t},y_{t}) \in \mathbb{F},
\end{array}
\end{equation}
with $({\bm{\pi_t}(\cdot), \, y_{t} })$ as the optimization variables. Clearly, \eqref{ex_optimal_probem} is an instance of \eqref{bare_optimal_probem}.

\subsection{Effects of Adjustable Invariant Sets}
The RMPC-AU \eqref{eq:sol}-\eqref{eq:cl_sys}, and \eqref{ex_optimal_probem} is solved with adjustable positive/control invariant sets and without any terminal constraints. Figure \ref{fig:O_set} illustrates the adjustable positive invariant set and the adjustable control invariant set computed using Algorithm \ref{algo_O_inf} and Algorithm \ref{algo_C_inf}, respectively.
To construct a closed loop system for the adjustable positive invariant set, we used some stabilizing affine state feedback policy, $\pi_{k\mid t} (x_{k\mid t}) = Kx_{k\mid t}$ obtained by pole-placement method. As expected, the sets in the state space shrink as the size of the uncertainty set, $y$, increases. 
In fact, when the uncertainty set is only $\{ 0\}$, i.e. $y=0$, it can be verified that the adjustable positive and control invariant sets are equivalent to the positive and control invariant sets for a deterministic system, respectively. Also, the  adjustable positive invariant set is much smaller than the adjustable control  invariant set because we fixed the control law to be affine for computing the positive adjustable invariant set, whereas the control adjustable invariant set allows any (also nonlinear) control policy.
\begin{figure}
	\centering
	\includegraphics[width=\linewidth,height=\textheight,keepaspectratio]{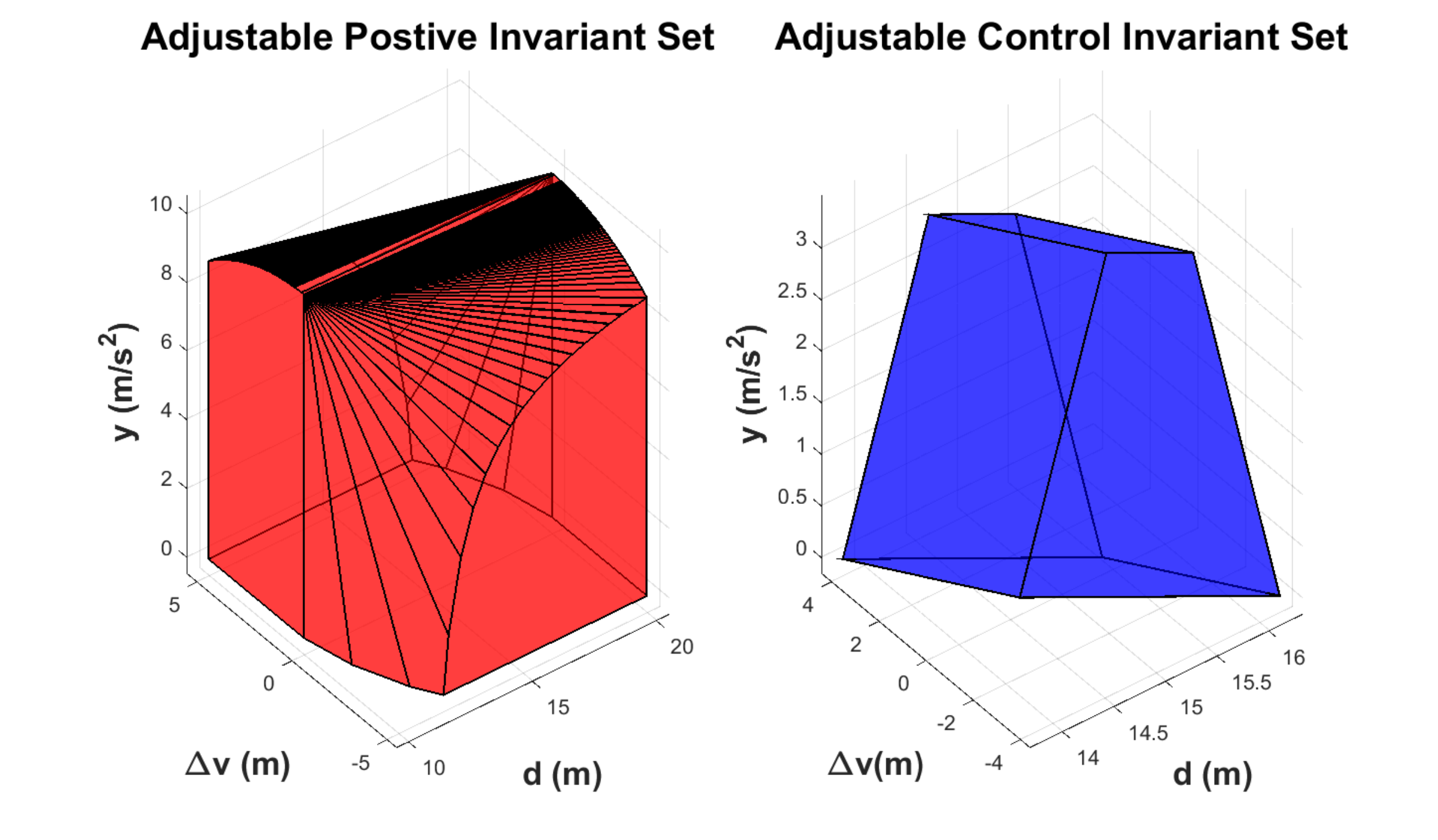}
	\caption{Adjustable Control Invariant Set (left) and Adjustable Positive Invariant Set (right)}
	\label{fig:O_set} 
\end{figure}  

In the closed loop simulation, RMPC-AU \eqref{eq:sol}-\eqref{eq:cl_sys}, and \eqref{ex_optimal_probem} was persistently feasible when the optimization problem \eqref{ex_optimal_probem} is solved with the adjustable control or positive invariant set. Note that even with the adjustable control invariant set the RMPC-AU was persistently feasible because our setting allows that $\{0\}$ is a subset of $\mathbb W$ at all times; see Remark \ref{rm1}.
It's also noted that with the adjustable positive invariant set, RMPC-AU was persistently feasible with the constant-size uncertainty set by virtue of Algorithm \ref{algo_O_inf}.

\subsection{Effects of $\lambda$}
Recall that $\lambda$ is the trade-off parameter between minimization of $d$ and the maximization of $y$. Fig.~\ref{fig:effec_lambda} shows the effect of $\lambda$ on the solution of the optimal control problem \eqref{ex_optimal_probem} with the terminal set as the positive adjustable invariant set in Fig.~\ref{fig:O_set}. The system dynamics is initialized at $d=15\si{m} $ and $\Delta v = 0\si{m/s} $. 
\begin{figure}
	\centering
	\includegraphics[width=\linewidth,height=\textheight,keepaspectratio]{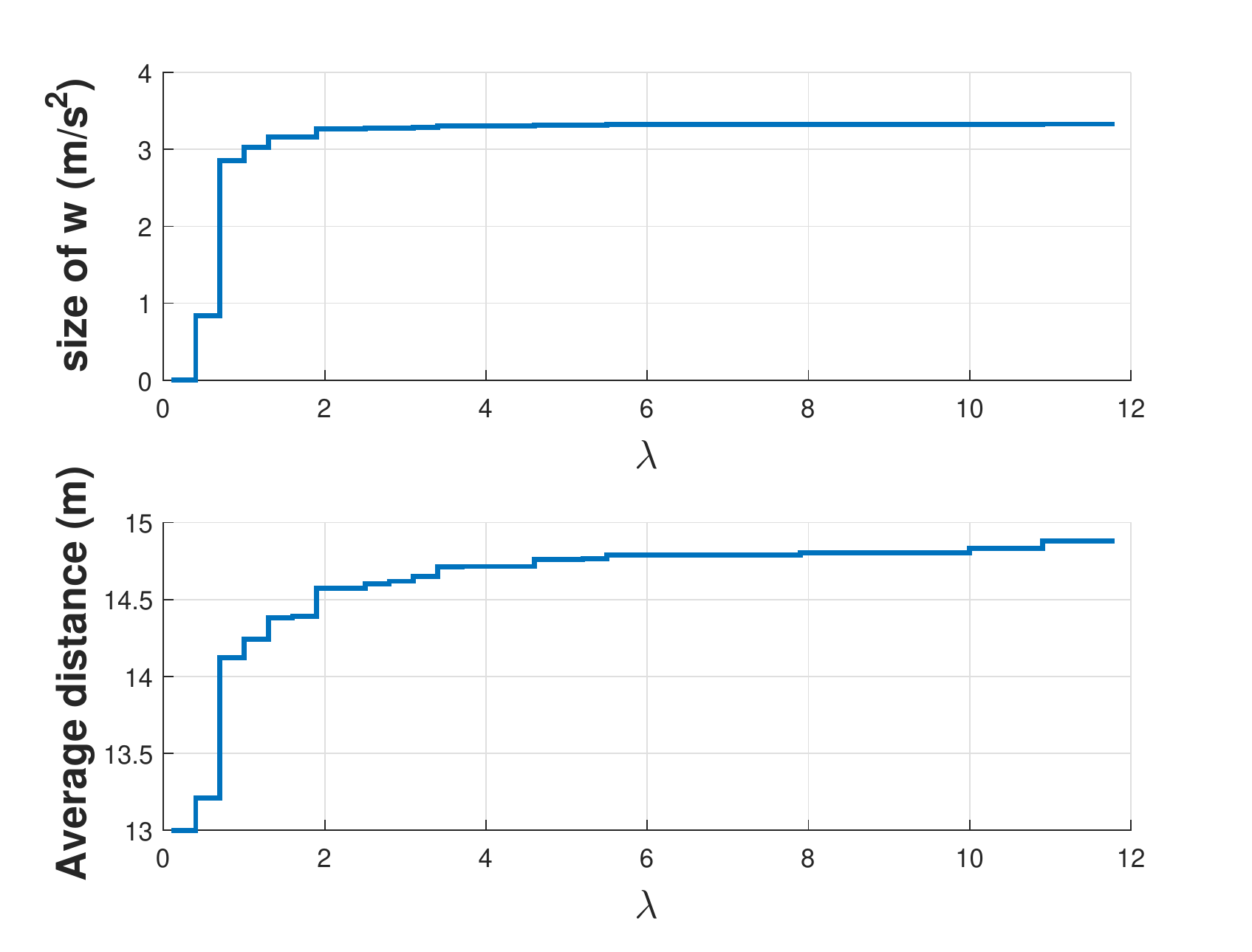}
	\caption{Effects of $\lambda$ on the size of $w$ (top) and average distance (bottom)}
	\label{fig:effec_lambda} 
\end{figure} 
In the upper plot in Fig.~\ref{fig:effec_lambda}, $y$ is plotted at different values of $\lambda$: higher values of $\lambda$ bring about a larger $y$, which is consistent with the formulation in~\eqref{ex_optimal_probem} and to be expected.
The lower plot in Fig.~\ref{fig:effec_lambda} shows, for different values of $\lambda$, the average distance when the front vehicle does not decelerate or accelerate; i.e., $y=0$. Also, $d$ increases as the value of $\lambda$ increases: since the front vehicle is allowed greater freedom in maneuvering, the tracking distance has to be correspondingly increased in order to preserve safety.

\section{Conclusions} \label{sec:conclusions}
	In this work, we studied the Robust Model Predictive Control problem with adjustable uncertainty sets. Our method optimizes over uncertainty sets as well as control policies, and guarantees robust recursive feasibility by enforcing an adjustable invariant set as a terminal constraint. We applied our framework to a cooperative adaptive cruise control and showed the validity of our approach in simulation. Future work includes new methods to construct adjustable invariant set for time-varying uncertainty sets and the actual experiment of our approach on CACC and examining our approach's effects on energy saving by adjusting the inter-vehicular distance.

	\section*{Appendix A: Tractable Reformulation of RMPC-AU}\label{append_lin_reform}
	The optimization problem \eqref{bare_optimal_probem} can be formulated into a computationally tractable convex problem through the following steps from the work of \cite{zhang2017robust}.
	First, we can reformulate our optimization problem \eqref{linear_adjusted_optimal_probem} using the polytopic terminal constraints \eqref{term_repren} as
	\begin{equation}
    \label{compact_linear_adjusted_optimal_probem}
    \begin{array}{ll}
    \displaystyle\min & \Big\{ \underset{ \mathbf{s_t} \in \mathbb{S}^N } {\textnormal{max}} \, \bm{c}^{\top}(\bm{P_ts_t+p_t})  \Big\} - \displaystyle\sum_{k=0}^{N-1}\lambda_k \rho (Y_{k\mid t} \mathbb{S} + y_{k\mid t}) \\
    \hphantom{.}\textrm{s.t.} & \\
    & \bm{P_t} \in \mathbb{R}^{Nn_u \times Nn_s}, \, \bm{p_t} \in \mathbb{R}^{Nn_u}, \, (\bm{Y_t},\bm{y_t}) \in \mathcal{Y}\,, \\
    & \bm{C_t(P_ts_t + p_t) } + \bm{D_t} (\bm{Y_t} \bm{s_t} + \mathbf{y_t}) \leq \bm{d_t}, \, \\
    & F \begin{bmatrix}
	M_t\bm{(P_ts_t + p_t) } + N_t (\bm{Y_t} \bm{s_t} + \bm{y_t}) \\
	\textnormal{vec}(Y_{N-1\mid t}) \\
	y_{N-1\mid t}]
	\end{bmatrix} \leq f,\\
	& \forall \, \bm{s_t} \in \mathbb S^N, 
    \end{array}
    \end{equation}
	where $( \bm{P_t}, \, \bm{p_t}, \, \bm{Y_t}, \, \bm{y_t})$ are the decision variables.
	$\bm{C_t} \in \mathbb R^{(n_f+n_g) \times Nn_u}$, $\bm{D_t}\in \mathbb R^{(n_f+n_g) \times Nn_w}$, and $\bm{d_t}\in \mathbb R^{(n_f+n_g)}$ together represent the state and input constraints \eqref{stateinputconstraints}; see the example \cite{goulart2006optimization}. Similar transformation is done to construct $M_t \in \mathbb R^{n_x \times Nn_u}$ and $N_t \in \mathbb R^{n_x \times Nn_w}$ to represent $x_{t+N\mid t}(\bm{s_t})$.
	
	Finally, by the duality argument of Ben-tal et al. \cite{ben2009robust}, the problem is reformulated into the finite-dimensional convex optimization problem \eqref{final_formulation}.

    \balance
    
    \section*{Appendix B: Adjustable control invariant set computation}
	In this section, we discuss the problem of computing adjustable control invariant set. Similar to the adjustable positive invariant set, the adjustable control invariant set can be in different shapes.
	In this section, we propose an algorithm to compute the special case of the adjustable control invariant set, $\mathbb{C}_{\textnormal{adj}}$, which can allow the same uncertainty at all time while guaranteeing the persistent feasibility of the RMPC-AU \eqref{bare_optimal_probem}-\eqref{eq:cl_sys}.
   
    Consider the linear system \eqref{linear_model}. Using the uncertainty restriction \eqref{eq:unceraintyset_prim}, the system dynamics at time $k$ can be re-written as: 
	\begin{align}
	x_{k+1} &= Ax_k + Bu_k + E(Ys_k+y) \nonumber \\
	&= Ax_k + Bu_k + E\mathcal{S}\textnormal{vec}(Y) + Ey,
    \end{align} 
    where $\mathcal{S} := \textnormal{diag}(s_k^T, ..., s_k^T) \in \mathbb{R}^{n_w \times n_wn_s}$.
    Finally, the parametric uncertainty model can be formulated as
    \begin{align} \label{eq:parametricuncertaintymodel_C}
            z_{k+1} =& \left[
            \begin{array}{c}
            x_{k+1}  \\
            \hline
            \textnormal{vec}(Y_{k+1}) \\
            \hline
            y_{k+1}
            \end{array}\right] = h(z_k,\, u_k,\,s_k) = \nonumber \\
            =& \left[
            \begin{array}{c|c|c}
            A & E \mathcal{S} & E   \\
            \hline
            0 & I & 0 \\
            \hline
            0 & 0 & I
            \end{array}\right]
            \left[
            \begin{array}{c}
            x_{k}  \\
            \hline
            \textnormal{vec}(Y_k) \\
            \hline
            y_k
            \end{array}\right]
            +
           \left[
        \begin{array}{c}
        B  \\
        \hline
        0 \\
        \hline
        0
        \end{array}\right] u_k.
    	\end{align}
    The robust control invariant set of the model \eqref{eq:parametricuncertaintymodel_C} represents the adjustable control invariant set of our system \eqref{linear_model}-\eqref{uncertaintycons} with the uncertainty set approximation \eqref{eq:unceraintyset_prim}.
    This set can be computed by Algorithm \ref{algo_C_inf} which uses the precursor set of the control system $\eqref{eq:parametricuncertaintymodel_C}$ instead of the autonomous system $\eqref{eq:parametricuncertaintymodel_O}$. In this case, a precursor set is defined by the following.
    \begin{definition}[Robust Precursor Set] \label{def_robust_precursor_C}
	For the system \eqref{eq:parametricuncertaintymodel_C}, we denote the robust precursor set to the set $\mathcal Z$ as
	\begin{equation}
	    \textnormal{Pre}(\mathcal Z, \, \mathbb S) = \{ z\in\mathbb R^n : \exists u\in \mathbb U \text{ s.t. } h(z,u,s)  \in \mathcal Z, \,\forall\, s \in \mathbb S \}.
	\end{equation}
	\end{definition}
	$\textnormal{Pre}(\mathcal Z, \, \mathbb S)$ defines the set of states $z$ which can be driven into the target set $\mathcal Z$ in one step while satisfying the system input constraints. 
\begin{algorithm}
\caption{Computation of $\mathbb{C}_{\textnormal{adj}}$}\label{algo_C_inf}
\begin{algorithmic}[1]
\State \textbf{Input} \textnormal{System model \eqref{eq:parametricuncertaintymodel_C}, $\mathcal{Z}$, $\mathbb{S}$}
\State \textbf{Output} \textnormal{$\mathbb{C}_{\textnormal{adj}}$}
\Indent
    \State $\Omega_0 \leftarrow \mathcal{Z}, \, k \leftarrow -1$
    \State \textbf{Repeat}
    \Indent
        \State $k \leftarrow k+1$, \, 
        \State $\Omega_{k+1} \leftarrow \textnormal{Pre}(\Omega_{k},\mathbb S) \, \cap \, \Omega_k$
    \EndIndent
    \State \textbf{Until} $\Omega_{k+1} = \Omega_k$
    \State $\mathbb{C}_{\textnormal{adj}} \leftarrow \Omega_k$
\EndIndent
\end{algorithmic}
\end{algorithm}
If Algorithm~\ref{algo_C_inf} converges, then $\mathbb{C}_{\textnormal{adj}}$ is an adjustable control invariant set.
	\bibliographystyle{plain}        
	\bibliography{library}

\begin{thebibliography}{10}

\bibitem{dorato1987historical}
P.~Dorato, ``A historical review of robust control,'' {\em IEEE Control Systems
  Magazine}, vol.~7, no.~2, pp.~44--47, 1987.

\bibitem{petersen2014robust}
I.~R. Petersen and R.~Tempo, ``Robust control of uncertain systems: Classical
  results and recent developments,'' {\em Automatica}, vol.~50, no.~5,
  pp.~1315--1335, 2014.

\bibitem{campo1987robust}
P.~J. Campo and M.~Morari, ``Robust model predictive control,'' in {\em
  American Control Conference, 1987}, pp.~1021--1026, IEEE, 1987.

\bibitem{low2000robust}
K.-S. Low and H.~Zhuang, ``Robust model predictive control and observer for
  direct drive applications,'' {\em IEEE Transactions on Power Electronics},
  vol.~15, no.~6, pp.~1018--1028, 2000.

\bibitem{wu2001lmi}
F.~Wu, ``{LMI}-based robust model predictive control and its application to an
  industrial {CSTR} problem,'' {\em Journal of process control}, vol.~11,
  no.~6, pp.~649--659, 2001.

\bibitem{zhang2014selling}
X.~Zhang, M.~Kamgarpour, P.~Goulart, and J.~Lygeros, ``Selling robustness
  margins: A framework for optimizing reserve capacities for linear systems,''
  in {\em Decision and Control (CDC), 2014 IEEE 53rd Annual Conference on},
  pp.~6419--6424, IEEE, 2014.

\bibitem{zhang2017robust}
X.~Zhang, M.~Kamgarpour, A.~Georghiou, P.~Goulart, and J.~Lygeros, ``Robust
  optimal control with adjustable uncertainty sets,'' {\em Automatica},
  vol.~75, pp.~249--259, 2017.

\bibitem{VrettosTPS2016}
E.~Vrettos, F.~Oldewurtel, and G.~Andersson, ``Robust energy-constrained
  frequency reserves from aggregations of commercial buildings,'' {\em IEEE
  Transactions on Power Systems}, vol.~31, pp.~4272--4285, Nov 2016.

\bibitem{Bitlisglioglu_TAC2017}
A.~Bitlislioğlu, T.~T. Gorecki, and C.~N. Jones, ``Robust tracking
  commitment,'' {\em IEEE Transactions on Automatic Control}, vol.~62,
  pp.~4451--4466, Sept 2017.

\bibitem{ReyZhang2018}
F.~Rey, X.~Zhang, S.~Merkli, V.~Agliati, M.~Kamgarpour, and J.~Lygeros,
  ``{Strengthening the Group: Aggregated Frequency Reserve Bidding with
  ADMM},'' {\em IEEE Transactions on Smart Grid}, 2018.

\bibitem{kouvaritakis2016model}
B.~Kouvaritakis and M.~Cannon, {\em Model predictive control}.
\newblock Springer, 2016.

\bibitem{lefevre2016learning}
S.~Lef{\`e}vre, A.~Carvalho, and F.~Borrelli, ``A learning-based framework for
  velocity control in autonomous driving,'' {\em IEEE Transactions on
  Automation Science and Engineering}, vol.~13, no.~1, pp.~32--42, 2016.

\bibitem{alam2014guaranteeing}
A.~Alam, A.~Gattami, K.~H. Johansson, and C.~J. Tomlin, ``Guaranteeing safety
  for heavy duty vehicle platooning: Safe set computations and experimental
  evaluations,'' {\em Control Engineering Practice}, vol.~24, pp.~33--41, 2014.

\bibitem{Guanetti2018}
J.~Guanetti, Y.~Kim, and F.~Borrelli, ``{Control of Connected and Automated
  Vehicles: State of the Art and Future Challenges},'' {\em Annual Reviews in
  Control}, 2018.

\bibitem{borrelli2017predictive}
F.~Borrelli, A.~Bemporad, and M.~Morari, {\em Predictive control for linear and
  hybrid systems}.
\newblock Cambridge University Press, 2017.

\bibitem{goulart2006optimization}
P.~J. Goulart, E.~C. Kerrigan, and J.~M. Maciejowski, ``Optimization over state
  feedback policies for robust control with constraints,'' {\em Automatica},
  vol.~42, no.~4, pp.~523--533, 2006.

\bibitem{gurobi}
I.~Gurobi~Optimization, ``Gurobi optimizer reference manual,'' 2016.

\bibitem{mosek2010mosek}
A.~Mosek, ``The mosek optimization software,'' {\em Online at http://www.
  mosek. com}, vol.~54, no.~2-1, p.~5, 2010.

\bibitem{turri2017cooperative}
V.~Turri, B.~Besselink, and K.~H. Johansson, ``Cooperative look-ahead control
  for fuel-efficient and safe heavy-duty vehicle platooning,'' {\em IEEE
  Transactions on Control Systems Technology}, vol.~25, no.~1, pp.~12--28,
  2017.

\bibitem{ben2009robust}
A.~Ben-Tal, L.~El~Ghaoui, and A.~Nemirovski, ``Robust optimization. princeton
  series in applied mathematics,'' 2009.

\end{thebibliography}

\end{document}